\documentclass[12pt]{amsart}

\usepackage{amsmath,amssymb,amsfonts,amsthm,mathrsfs,epsfig}
\usepackage{graphicx}
\usepackage[usenames,dvipsnames]{color}
\usepackage{hyperref}
\usepackage{array}

\usepackage[margin=1in]{geometry}

\newcommand{\pt}[1]{{\color{Red} #1}}

\begin{document}
\newtheorem{theorem}{Theorem}[section]
\newtheorem{lemma}[theorem]{Lemma}
\newtheorem{definition}[theorem]{Definition}
\newtheorem{conjecture}[theorem]{Conjecture}
\newtheorem{proposition}[theorem]{Proposition}
\newtheorem{algorithm}[theorem]{Algorithm}
\newtheorem{corollary}[theorem]{Corollary}
\newtheorem{observation}[theorem]{Observation}
\newtheorem{claim}[theorem]{Claim}
\newtheorem{problem}[theorem]{Open Problem}
\newtheorem{remark}[theorem]{Remark}
\newcommand{\noin}{\noindent}
\newcommand{\ind}{\indent}
\newcommand{\om}{\omega}
\newcommand{\I}{\mathcal I}
\newcommand{\N}{{\mathbb N}}
\newcommand{\Z}{{\mathbb Z}}
\newcommand{\LL}{\mathbb{L}}
\newcommand{\R}{{\mathbb R}}
\newcommand{\E}[1]{\mathbb{E}\left[#1 \right]}
\newcommand{\V}{\mathbb Var}
\newcommand{\Prob}{\mathbb{P}}
\newcommand{\eps}{\varepsilon}
\newcommand{\bsigma}{{\boldsymbol\sigma}}
\newcommand{\bm}{{\boldsymbol m}}
\newcommand{\bu}{{\boldsymbol u}}
\newcommand{\bv}{{\boldsymbol v}}
\newcommand{\tU}{{\mathtt U}}
\newcommand{\tD}{{\mathtt D}}
\newcommand{\tL}{{\mathtt L}}
\newcommand{\tR}{{\mathtt R}}

\newcommand{\Tv}{P}

\newcommand{\mT}{\mathcal{T}}
\newcommand{\mS}{\mathcal{S}}
\newcommand{\mA}{\mathcal{A}}
\newcommand{\mB}{\mathcal{B}}

\newcommand{\Cyc}[1]{\mathrm{Cyc}\left(#1\right)}
\newcommand{\Seq}[1]{\mathrm{Seq}\left(#1\right)}
\newcommand{\Mul}[1]{\mathrm{Mul}\left(#1\right)}
\newcommand{\Mulo}[1]{\mathrm{Mul}_{>0}\left(#1\right)}
\newcommand{\Set}[1]{\mathrm{Set}\left(#1\right)}
\newcommand{\Setd}[1]{\mathrm{Set}_{d}\left(#1\right)}
\newcommand{\Bin}{\mathrm{Bin}}

\newcommand{\remove}[1]{}

\newcommand{\circled}[1]{\Large \textcircled{\normalsize #1}}

\title[How many zombies are needed to catch the survivor on toroidal grids?]{How many zombies are needed to catch the survivor on toroidal grids? Embarrassing upper bound of $O(n^2)$.}

\author{Pawe\l{} Pra\l{}at}
\address{Department of Mathematics, Ryerson University, Toronto, ON, Canada}
\email{\tt pralat@ryerson.ca}

\thanks{The author gratefully acknowledges support from NSERC}

\begin{abstract}
In Zombies and Survivors, a set of zombies attempts to eat a lone survivor loose on a given graph. The zombies randomly choose their initial location, and during the course of the game, move directly toward the survivor. At each round, they move to the neighbouring vertex that minimizes the distance to the survivor; if there is more than one such vertex, then they choose one uniformly at random. The survivor attempts to escape from the zombies by moving to a neighbouring vertex or staying on his current vertex. The zombies win if eventually one of them eats the survivor by landing on their vertex; otherwise, the survivor wins. The zombie number of a graph is the minimum number of zombies needed to play such that the probability that they win is at least 1/2. 

This variant of the game was recently investigated for several graph families, such as cycles, hypercubes, incidence graphs of projective planes, and grids $P_n \square P_n$. However, unfortunately, still very little is known for toroidal grids $C_n \square C_n$: the zombie number of $C_n \square C_n$ is at least $\sqrt n/(\omega\log n)$, where $\omega = \omega(n)$ is any function going to infinity as $n \to \infty$, and no upper bound is known except a trivial bound of $O(n^2 \log n)$. In this note, we provide an approach that gives an embarrassing bound of $O(n^2)$ but it is possible that (with more careful, deterministic, argument) it might actually give a bound of $O(n^{3/2})$. On the other hand, by analyzing a specific strategy for the survivor, it seems that one could slightly improve the lower bound to $\sqrt n/\omega$. In any case, we are far away from understanding this intriguing question. \pt{Your help is needed!}
\end{abstract}

\maketitle

\section{Introduction}

For a given connected graph $G$ and given $k \in \N$, we consider the following probabilistic variant of Cops and Robbers, which is played over a series of discrete time-steps. In the game of \emph{Zombies and Survivors}, suppose that $k$ \emph{zombies} (akin to the cops) start the game on random vertices of $G$; each zombie, independently, selects a vertex uniformly at random to start with. Then the \emph{survivor} (akin to the robber) occupies some vertex of $G$. As zombies have limited intelligence, in each round, a given zombie moves towards the survivor along a shortest path connecting them. In particular, the zombie decreases the distance from its vertex to the survivor's. If there is more than one neighbour of a given zombie that is closer to the survivor than the zombie is, then they move to one of these chosen uniformly at random. Each zombie moves independently of all other zombies. As in Cops and Robbers, the survivor may move to another neighbouring vertex, or \emph{pass} and not move. The zombies win if one or more of them \emph{eat} the survivor; that is, land on the vertex which the survivor currently occupies. The survivor, as survivors should do in the event of a zombie attack, attempts to survive by applying an optimal strategy; that is, a strategy that minimizes the probability of being captured. Note that there is no strategy for the zombies; they merely move on geodesics towards the survivor in each round. Note that since zombies always move toward the survivor, he can pass at most $D$ times, where $D$ is a diameter of $G$, before being eaten by some zombie. We note also that our probabilistic version of Zombies and Survivors was inspired by a deterministic version of this game (with similar rules, but the zombies may choose their initial positions, and also choose which shortest path to the survivor they will move on) first considered in~\cite{hm}.

\bigskip

Let $s_k(G)$ be the probability that the survivor wins the game, provided that he follows the optimal strategy. Clearly, $s_k(G)=1$ for $k < c(G)$, where $c(G)$ is the cop number of $G$. On the other hand, $s_k(G) < 1$ provided that there is a strategy for $k \ge c(G)$ cops in which the cops always try to get closer to the robber, since with positive probability the zombies may follow such a strategy. Usually, $s_k(G) > 0$ for any $k \ge c(G)$; however, there are some examples of graphs for which $s_k(G) = 0$ for every $k \ge c(G)$ (consider, for example, trees). Further, note that $s_k(G)$ is a non-decreasing function of $k$ (that is, for every $k \ge 1$, $s_{k+1}(G) \le s_k(G)$), and $s_k(G) \to 0$ as $k\to \infty$. The latter limit follows since the probability that each vertex is initially occupied by at least one zombie tends to 1 as $k \to \infty$.

Define the \emph{zombie number} of a graph $G$ by
$$
z(G) = \min \{ k \ge c(G) : s_k(G) \le 1/2 \} .
$$
This parameter is well defined since $\lim_{k\to\infty}s_k(G)=0$. In other words, $z(G)$ is the minimum number of zombies such that the probability that they eat the survivor is at least 1/2. The ratio $Z(G) = z(G)/c(G) \ge 1$ is the \emph{cost of being undead}.  Note that there are examples of families of graphs for which there is no cost of being undead; that is, $Z(G)=1$ (as is the case if $G$ is a tree), and, as argued in~\cite{zombies}, there are examples of graphs with $Z(G)=\Theta(n).$

\bigskip

This variant of the game was recently introduced in~\cite{zombies}, where several graph families were investigated, such as cycles, hypercubes, incidence graphs of projective planes, and Cartesian grids. In particular, the zombie number of the incidence graphs of projective planes is about two times larger than the corresponding cop number; for hypercubes, this ratio is asymptotically $4/3$. 

It seems that the main question for this game is to investigate the zombie number for grids formed by products of cycles, $T_n = C_n \square C_n$ (so called \emph{toroidal grids}). In~\cite{zombies}, it was proved that the zombie number of $T_n$ is at least $\sqrt n/(\omega\log n)$, where $\omega = \omega(n)$ is any function going to infinity as $n \to \infty$. The proof relies on the careful analysis of a strategy for the survivor. On the other hand, trivially, $z(T_n) = O(n^2 \log n)$. Suppose that the game is played against $k = 3 n^2 \log n$ zombies. It is easy to see that a.a.s.\ every vertex is initially occupied by at least one zombie and if so the survivor is eaten immediately. Indeed, the probability that at at least one vertex is not not occupied by a zombie is at most $n^2 (1-1/n^2)^k \le n^2 \exp(-k/n^2) = 1/n = o(1)$. 

In this paper, we provide a general approach that gives immediately a tiny improvement, an upper bound of $O(n^2)$. We investigate this technique by analyzing a few natural strategies and show that this approach has no hope of giving an upper bound better than $O(n^{3/2})$. However, it is quite possible that (with more careful, deterministic, argument) it might actually give that bound. It is left for a further investigation. On the other hand, by analyzing a specific strategy for the survivor, it seems that one could slightly improve the lower bound to $\sqrt n/\omega$. In any case, we are far away from understanding this intriguing question. \pt{Your help is needed!}

\section{Results}

\subsection{Definitions} 

Let us start with a formal definition of the class of graphs we investigate in this paper.
For graphs $G$ and $H$, define the \emph{Cartesian product} of $G$ and $H$, written $G\square H,$ to have vertices $V(G)\times V(H),$ and vertices $(a,b)$ and $(c,d)$ are joined if $a=c$ and $bd \in E(H)$ or $ac \in E(G)$ and $b=d.$ In this note, we consider grids formed by products of cycles.  Let $T_n$ be the \emph{toroidal grid} $n\times n$, which is isomorphic to $C_n \square C_n.$ For simplicity, we take the vertex set of $T_n$ to consist of $\Z_n \times\Z_n$, where $\Z_n$ denotes the ring of integers modulo $n$.

Results in the paper are asymptotic in nature as $n \rightarrow \infty$. We emphasize that the notations $o(\cdot)$ and $O(\cdot)$ refer to functions of $n$, not necessarily positive, whose growth is bounded. We say that an event in a probability space holds \emph{asymptotically almost surely} (or \emph{a.a.s.}) if the probability that it holds tends to $1$ as $n$ goes to infinity. Finally, for simplicity we will write $f(n) \sim g(n)$ if $f(n)/g(n) \to 1$ as $n \to \infty$; that is, when $f(n) = (1+o(1)) g(n)$.

\subsection{General upper bound} 

We are going to play the game on the toroidal grid $T_n$. Let us consider the family $\mathcal{F}$ of strategies of the survivor for the first $M=M(n):=\lfloor n/4 \rfloor$ moves of the game. Here, a \emph{strategy} is simply a sequence of moves of the survivor which is not affected neither by initial distribution of zombies nor by their behaviour during the game. Moreover, we may assume that zombies do not eat the survivor immediately when they catch him; instead, they walk with him to the end of this sequence of $T$ moves and then do their job. Hence, since there are $n^2$ vertices to choose from for the starting position of the survivor and 5 options in each round (``go west'', ``go east'', ``go north'', ``go south'', and ``stay put''), the number of strategies in $\mathcal{F}$ is $n^2 5^{M} = n^2 5^{\lfloor n/4 \rfloor}$. Finally, let $\mathcal{F}_0 \subseteq \mathcal{F}$ be a subfamily of $5^{\lfloor n/4 \rfloor}$ strategies of the survivor that finish his walk at vertex $(0,0)$ of $T_n$.

Let us concentrate on a given strategy $\mathcal{S} \in \mathcal{F}$ finishing at vertex $(a,b)$ of $T_n$. For any $x,y \in \Z$ such that $|x| + |y| \le M$, let $p_{\mathcal{S}}(x,y)$ be the probability that a zombie starting at vertex $(a+x,b+y)$ eats the survivor using strategy $\mathcal{S}$ for the first $M$ moves. Clearly, if a zombie starts at vertex at distance larger than $M$ from $(a,b)$, then it is impossible for her to catch the survivor (even intelligent player would not be able to reach $(a,b)$ in $M$ moves!); hence, in this situation $p_{\mathcal{S}}(x,y) = 0$. Note also that restricting staring positions for zombies to the subgraph around $(a,b)$ guarantees that the survivor and zombies starting at this subgraph do not leave it during the fist $M$ rounds. As a result, the game during these $M$ rounds is played as if it was played on the square grid ($P_n \square P_n$) centred at $(a,b)$, not the toroidal one. Moreover, since $T_n$ is a vertex transitive graph, without loss of generality, we may assume that the survivor finishes his walk at vertex $(0,0)$. These little observations will simplify the analysis below. Finally, let 
$$
t(\mathcal{S}) = \sum_{ x = - M}^{M}  \sum_{ y = - M}^{M} p_\mathcal{S}(x,y).
$$
As mentioned earlier, due to the fact that $T_n$ is vertex transitive, we have
$$
t_n := \min_{\mathcal{S} \in \mathcal{F}} t(\mathcal{S}) = \min_{\mathcal{S} \in \mathcal{F}_0} t(\mathcal{S}).
$$

Now, we are ready to state our first observation.

\begin{theorem}\label{thm:gen_bound}
A.a.s.\ $k = n^3 / t_n$ zombies catch the survivor on $T_n$. Hence, $z(T_n) = O(n^3 / t_n)$.
\end{theorem}
\begin{proof}
Let $X_1, X_2, \ldots$ be a sequence of independent random variables, each of them being the Bernoulli random variable with parameter $p=1/2$ (that is, $\Pr(X_i = 1) = \Pr(X_i = 0) = 1/2$ for each $i \in \N$). This (random) sequence will completely determine the behaviour of all the zombies. Formally, let us first fix a permutation $\pi$ of $k$ zombies. Then, in each round, we consider all zombies, one by one, using permutation $\pi$. If there is precisely one shortest path between a given zombie and the survivor, the next move is determined; otherwise, the next random variable $X_i$ from the sequence guides it (for example, if $X_i = 0$, the zombie moves horizontally; otherwise, she moves vertically). 

Our goal is to show that a.a.s.\ the survivor is eaten during the fist $M = \lfloor n/4 \rfloor$ rounds, regardless of the strategy he uses. In order to show it, let us pretend that the game is played on a \emph{real board} but, at the same time, there are $n^2 5^M$ \emph{auxiliary boards} where the game is played against all strategies from $\mathcal{F}$. An important assumption is that the same sequence $X_1, X_2, \ldots$ is used for all the boards. 

Since zombies select initial vertices uniformly at random, the probability that a given zombie wins against a given strategy $\mathcal{S} \in \mathcal{F}$ is at least
$$
\sum_{ x = -M}^{M}  \sum_{ y = - M}^{M} \frac {p_\mathcal{S}(x,y)}{n^2} \ge \frac {t_n}{n^2}.
$$
Since zombies play independently, the probability the survivor using strategy $\mathcal{S}$ is not eaten during the first $M$ rounds is at most
$$
\left( 1 - \frac {t_n}{n^2} \right)^k \le \exp \left( - \frac {t_n k}{n^2} \right) = \exp( - n) = o \left (\frac {1}{n^2 5^{n/4}} \right).
$$
(Note that, in particular, if two zombies start at the same vertex $(x,y)$, each of them catches the survivor with probability $p_{\mathcal{S}}(x,y)$ and the corresponding events are independent.) It follows that the expected number of auxiliary games where the survivor is not eaten is $o(1)$ and so a.a.s.\ he loses on all auxiliary boards. As a result, a.a.s.\ zombies win the real game too, regardless of the strategy used by the survivor. Clearly, the survivor should make decisions based on the behaviour of the zombies (who make decisions at random and so very often behave in a sub-optimal way); however, if the survivor wins the real game, at least one auxiliary game is also won by some specific strategy which we showed cannot happen a.a.s. That is why we needed the trick with the sequence $X_1, X_2, \ldots$ guiding all the games considered.
\end{proof}

\subsection{Deriving $p_{\mathcal{S}}(x,y)$} 

In this subsection, we first derive a recursive formula for calculating $p_{\mathcal{S}}(x,y)$ for a given strategy $\mathcal{S} \in \mathcal{F}$. Consider any strategy $\mathcal{S} \in \mathcal{F}$: at the beginning of round $i$ ($1 \le i \le M = \lfloor n/4 \rfloor$), the survivor occupies vertex $(x_i, y_i)$, zombies make their moves and then the survivor moves to $(x_{i+1}, y_{i+1})$; round $i$ is finished and round $i+1$ starts. As before, since $T_n$ is vertex transitive, without loss of generality we may restrict ourselves to a subfamily $\mathcal{F}_0$ of strategies for the first $M$ moves that bring the survivor to vertex $(0,0)$ after $M$ rounds; that is, $(x_M,y_M)=(0,0)$. We are interested in calculating $p_{\mathcal{S}}^i(x,y)$, the probability that a zombie occupying $(x,y)$ at the beginning of round $i$ catches the survivor using strategy $\mathcal{S}$ by the end of round $M$. Note that $p_{\mathcal{S}}(x,y) = p_{\mathcal{S}}^{1}(x,y)$. Clearly,
$$
p_{\mathcal{S}}^M(0,0) = p_{\mathcal{S}}^M(1,0) = p_{\mathcal{S}}^M(-1,0) = p_{\mathcal{S}}^M(0,1) = p_{\mathcal{S}}^M(0,-1) = 1,
$$
and $p_{\mathcal{S}}^M(x,y) = 0$ if $|x|+|y| > 1$ (zombie must be at distance at most 1 from the survivor at the beginning of round $M$). Moreover, for any round $1 \le i < M$ we have the following recursive formula. If $x \neq x_i$ and $y \neq y_i$, then
$$
p_{\mathcal{S}}^{i}(x,y) = \frac {p_{\mathcal{S}}^{i+1}(x \pm 1,y)}{2} + \frac {p_{\mathcal{S}}^{i+1}(x,y \pm 1)}{2},
$$
where ``$\pm$'' is ``$+$'' or ``$-$'' so that $|x_i - (x \pm 1)| < |x_i-x|$ and $|y_i - (y \pm 1)| < |y_i-y|$ (zombies move towards the survivor, either horizontally or diagonally; the decision is made uniformly at random). If $x = x_i$ and $y \neq y_i$, then
$$
p_{\mathcal{S}}^{i}(x,y) = p_{\mathcal{S}}^{i+1}(x,y \pm 1)
$$
(zombies move diagonally towards the survivor). Finally, if $x \neq x_i$ and $y = y_i$, then
$$
p_{\mathcal{S}}^{i}(x,y) = p_{\mathcal{S}}^{i+1}(x \pm 1,y)
$$
(zombies move horizontally towards the survivor). 

\begin{figure}
\begin{center}
\includegraphics[width=0.5\textwidth]{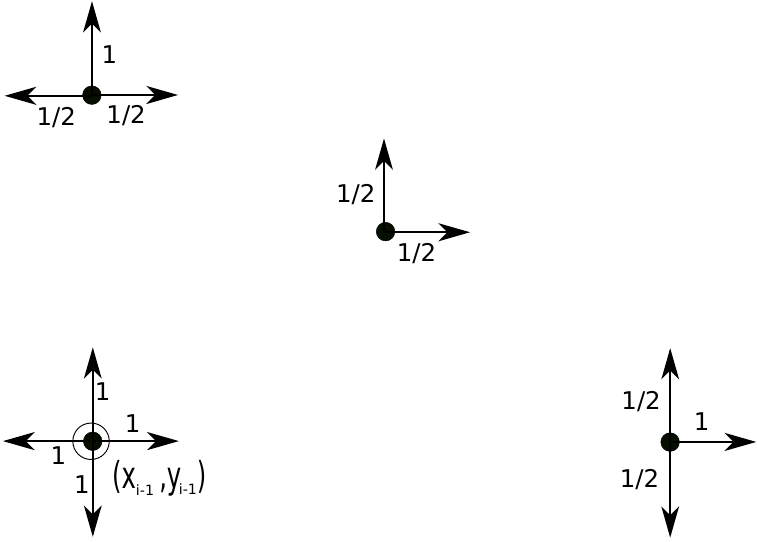} 
\end{center}
\caption{Getting $p_{\mathcal{S}}^{i-1}(x,y)$ from $p_{\mathcal{S}}^{i}(x,y)$.}
\label{fig:spreading}
\end{figure}

Reversing the direction, we get the following recursive relationship between $p_{\mathcal{S}}^{i}(x,y)$ and $p_{\mathcal{S}}^{i-1}(x,y)$. (See Figure~\ref{fig:spreading} for an illustration.) If $x \neq x_{i-1}$ and $y \neq y_{i-1}$, then half of the weight of $p_{\mathcal{S}}^{i}(x,y)$ is moved horizontally and half of it is moved vertically to the two neighbours of $(x,y)$ that are further away from $(x_{i-1},y_{i-1})$. If $x \neq x_{i-1}$ but $y = y_{i-1}$, then half of the weight of $p_{\mathcal{S}}^{i}(x,y)$ is moved vertically to both ``vertical'' neighbours of $(x,y)$; in addition to this, $p_{\mathcal{S}}^{i}(x,y)$ is added to the ``horizontal'' neighbour of $(x,y)$ that is further away from $(x_{i-1},y_{i-1})$. We proceed in a symmetric way if $x = x_{i-1}$ and $y \neq y_{i-1}$. Finally, we put
\begin{eqnarray*}
p_{\mathcal{S}}^{i-1}(x_{i-1},y_{i-1}) &=& p_{\mathcal{S}}^i(x_{i-1}+1,y_{i-1}) = p_{\mathcal{S}}^i(x_{i-1}-1,y_{i-1}) \\
&=& p_{\mathcal{S}}^i(x_{i-1},y_{i-1}+1) = p_{\mathcal{S}}^i(x_{i-1},y_{i-1}-1) = 1.
\end{eqnarray*}

\begin{table}
\begin{center}
{\tiny
\begin{tabular}{ccccccc}
&&&&&&\\
&&&&&&\\
&&&1.00&&&\\
&&1.00&{\color{Red}1.00}&1.00&&\\
&&&1.00&&&\\
&&&&&&\\
&&&&&&\\
\end{tabular}
$\to$
\begin{tabular}{ccccccc}
&&&&&&\\
&&&1.00&&&\\
&&1.00&1.00&1.00&&\\
&1.00&1.00&{\color{Red}1.00}&1.00&1.00&\\
&&1.00&1.00&1.00&&\\
&&&1.00&&&\\
&&&&&&\\
\end{tabular}
$\to$
\begin{tabular}{ccccccc}
&&&1.00&&&\\
&&1.00&1.00&1.00&&\\
&1.00&1.00&1.00&1.00&1.00&\\
1.00&1.00&1.00&1.00&1.00&1.00&1.00\\
&1.00&1.00&1.00&1.00&1.00&\\
&&1.00&1.00&1.00&&\\
&&&1.00&&&\\
\end{tabular}
}
\caption{The survivor does not move: $p_{\mathcal{S}}^M \to p_{\mathcal{S}}^{M-1} \to p_{\mathcal{S}}^{M-2}$}
\label{tab:not_moving}
\end{center}
\end{table}

\begin{table}
\begin{center}
{\tiny
\begin{tabular}{ccccccc}
&&&&&&\\
&&&&&&\\
&&&1.00&&&\\
&&1.00&1.00&1.00&&\\
&&&{\color{Red}1.00}&&&\\
&&&&&&\\
&&&&&&\\
\end{tabular}
$\to$
\begin{tabular}{ccccccc}
&&&&&&\\
&&&1.00&&&\\
&&1.00&1.00&1.00&&\\
&0.50&0.50&1.00&0.50&0.50&\\
&&1.00&1.00&1.00&&\\
&&&{\color{Red}1.00}&&&\\
&&&&&&\\
\end{tabular}
$\to$
\begin{tabular}{ccccccc}
&&&1.00&&&\\
&&1.00&1.00&1.00&&\\
&0.75&0.75&1.00&0.75&0.75&\\
0.25&0.25&1.00&1.00&1.00&0.25&0.25\\
&0.50&0.50&1.00&0.50&0.50&\\
&&1.00&1.00&1.00&&\\
&&&1.00&&&\\
\end{tabular}
}
\caption{The survivor goes straight up: $p_{\mathcal{S}}^M \to p_{\mathcal{S}}^{M-1} \to p_{\mathcal{S}}^{M-2}$}
\label{tab:go_straight}
\end{center}
\end{table}

Table~\ref{tab:not_moving} presents the first two iterations for the strategy where the survivor does not move at all. Of course, for this simple strategy there is no need to use recursive formula. Clearly, $p_{\mathcal{S}}^1(x,y) = 1$ if the distance from $(x,y)$ to $(0,0)$ is at most $M$ (that is, $|x|+|y| \le M$) and $p_{\mathcal{S}}^1(x,y) = 0$ otherwise. The next example, presented in Table~\ref{tab:go_straight} is more interesting; this time the survivor goes straight up. In both cases, vertex $(x_{i-1}, y_{i-1})$ that is used to get $p_{\mathcal{S}}^{i-1}$ from $p_{\mathcal{S}}^{i}$ is coloured red. 

\subsection{Upper bound: $z(T_n) = O(n^2)$}

Since not only all the weight contributing to $\sum_{x,y} p_{\mathcal{S}}^i(x,y)$ is preserved in $\sum_{x,y} p_{\mathcal{S}}^{i-1}(x,y)$ but each time the total weight increases by at least 4, it trivially follows that $t_n \ge 4M = 4 \lfloor n/4 \rfloor \sim n$. As a result, from Theorem~\ref{thm:gen_bound} we get the following upper bound for $z(T_n)$. 

\begin{corollary}
$z(T_n) = O(n^2)$.
\end{corollary}

This improves the trivial upper bound of $O(n^2 \log n)$ but it is embarrassing that this is the best we can do. 
Is there an upper bound of $n^{2 - \eps}$ for some $\eps>0$? This remains an open question. 

\subsection{Examples of 8 strategies} 

\begin{figure}
\begin{center}
\begin{tabular}{cc}
\includegraphics[width=0.4\textwidth]{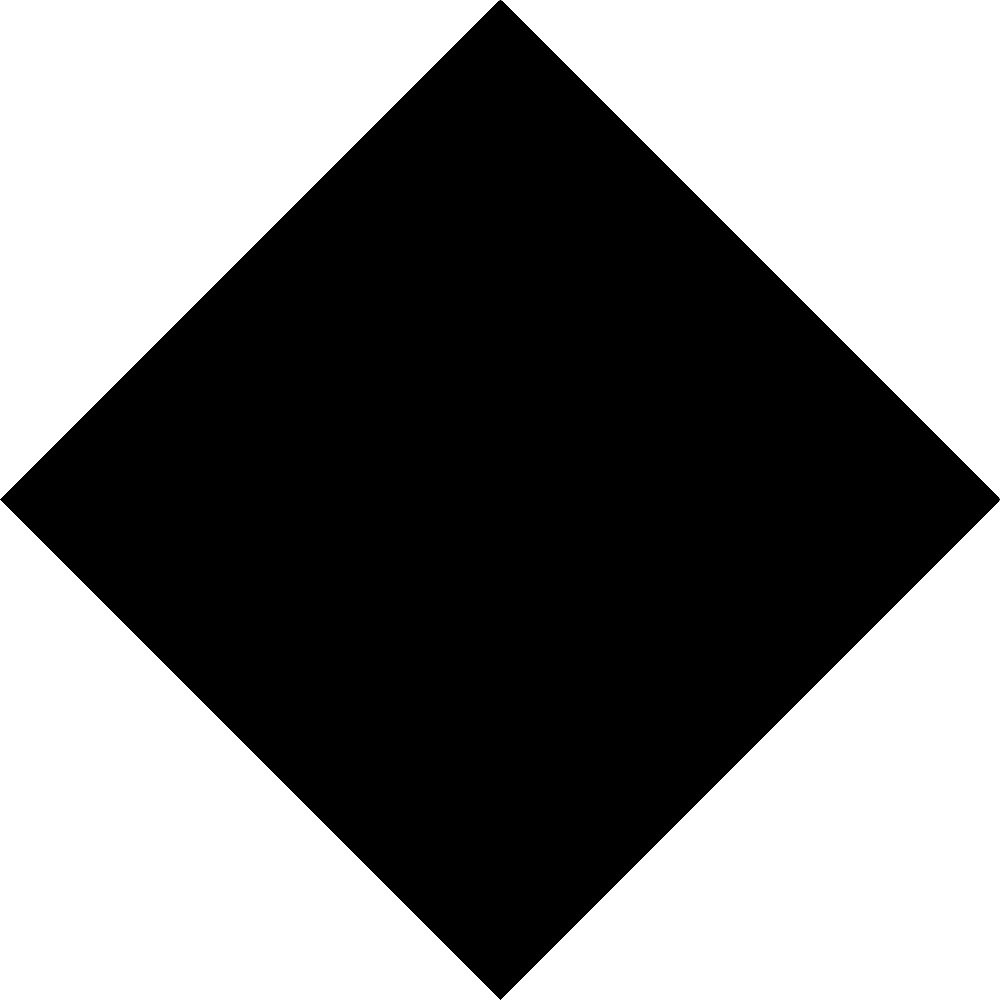} & \includegraphics[width=0.4\textwidth]{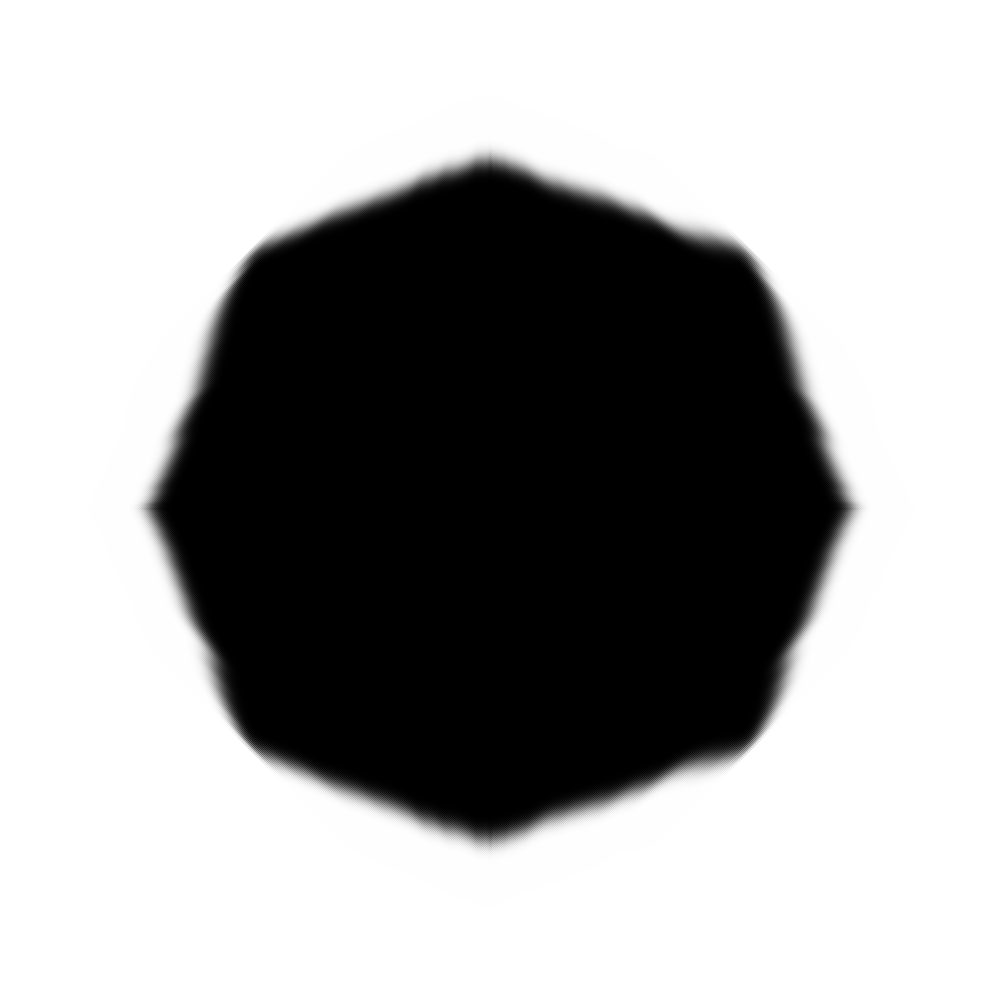} \\ 
(a) stay put & (b) move randomly \\
\includegraphics[width=0.4\textwidth]{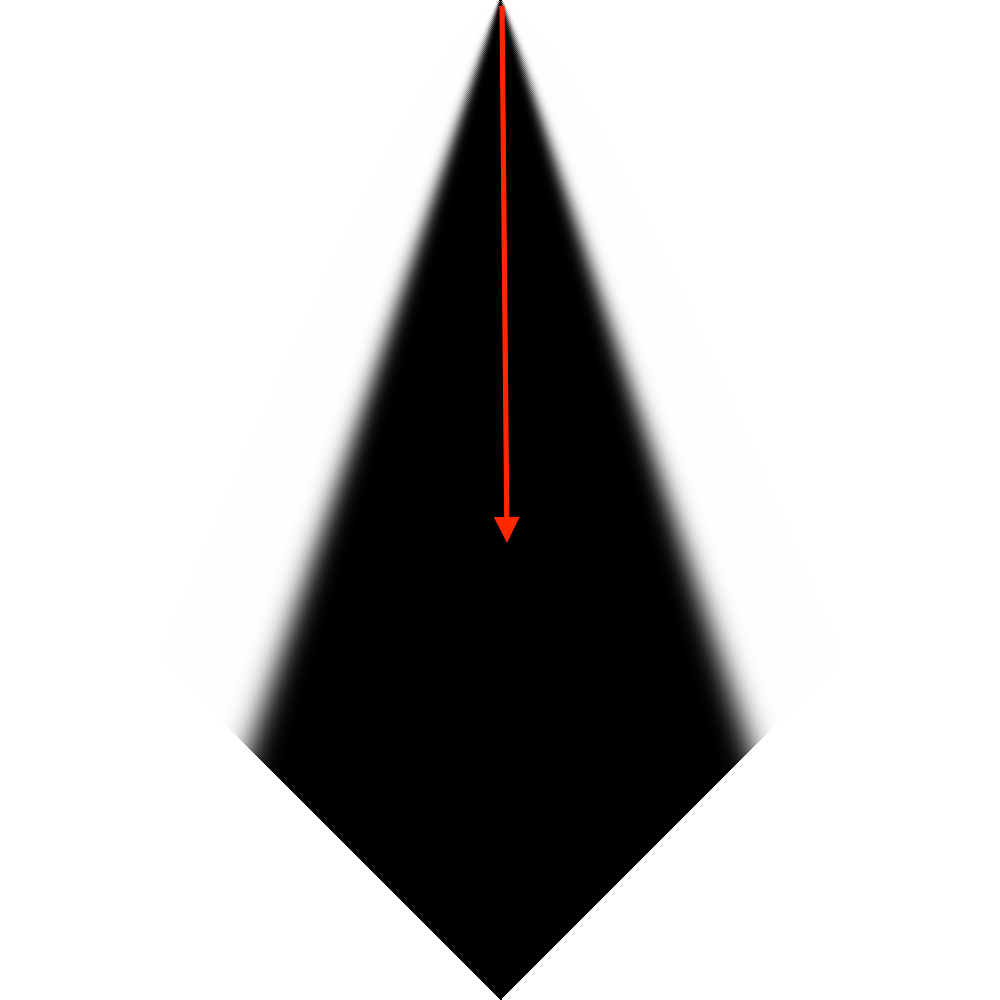} & \includegraphics[width=0.4\textwidth]{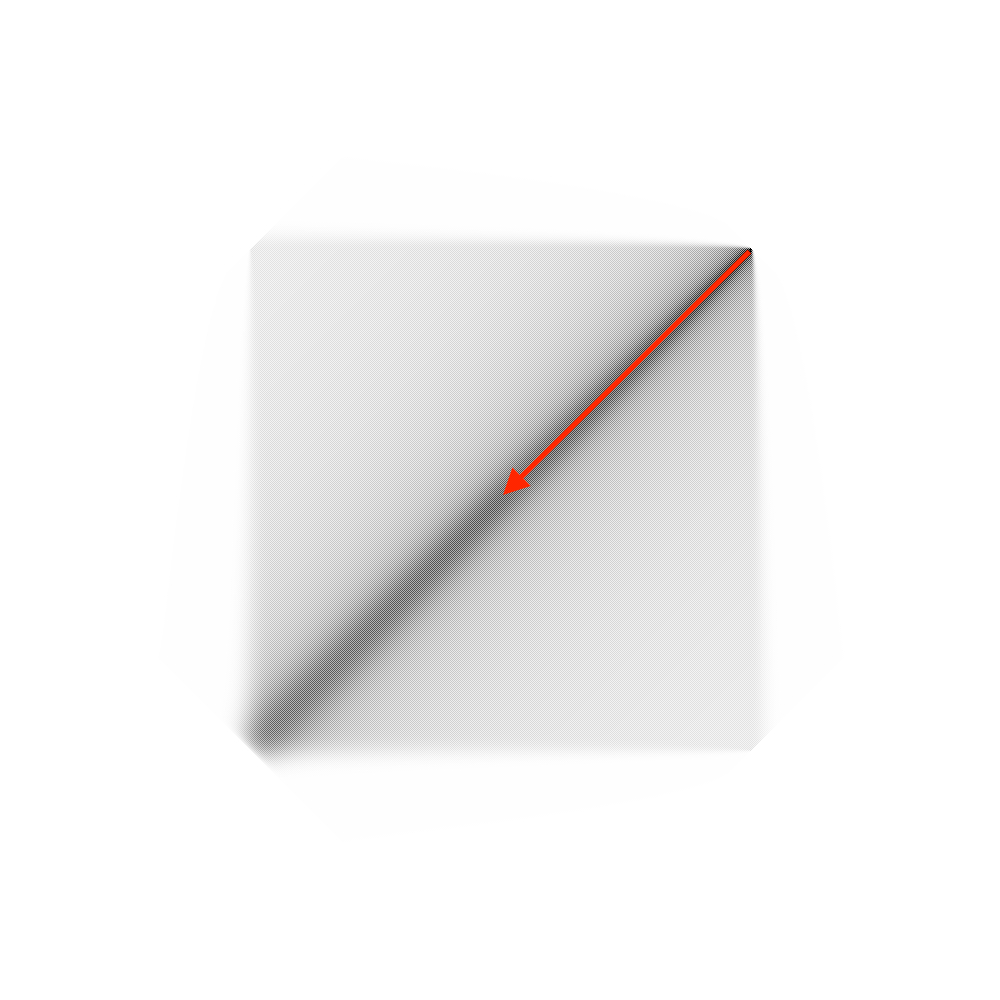} \\ 
(c) go down & (d) go diagonally \\
\end{tabular}
\end{center}
\caption{Examples of 4 strategies for the first $M$ moves}
\label{fig:4strategies}
\end{figure}

\begin{figure}
\begin{center}
\begin{tabular}{cc}
\includegraphics[width=0.4\textwidth]{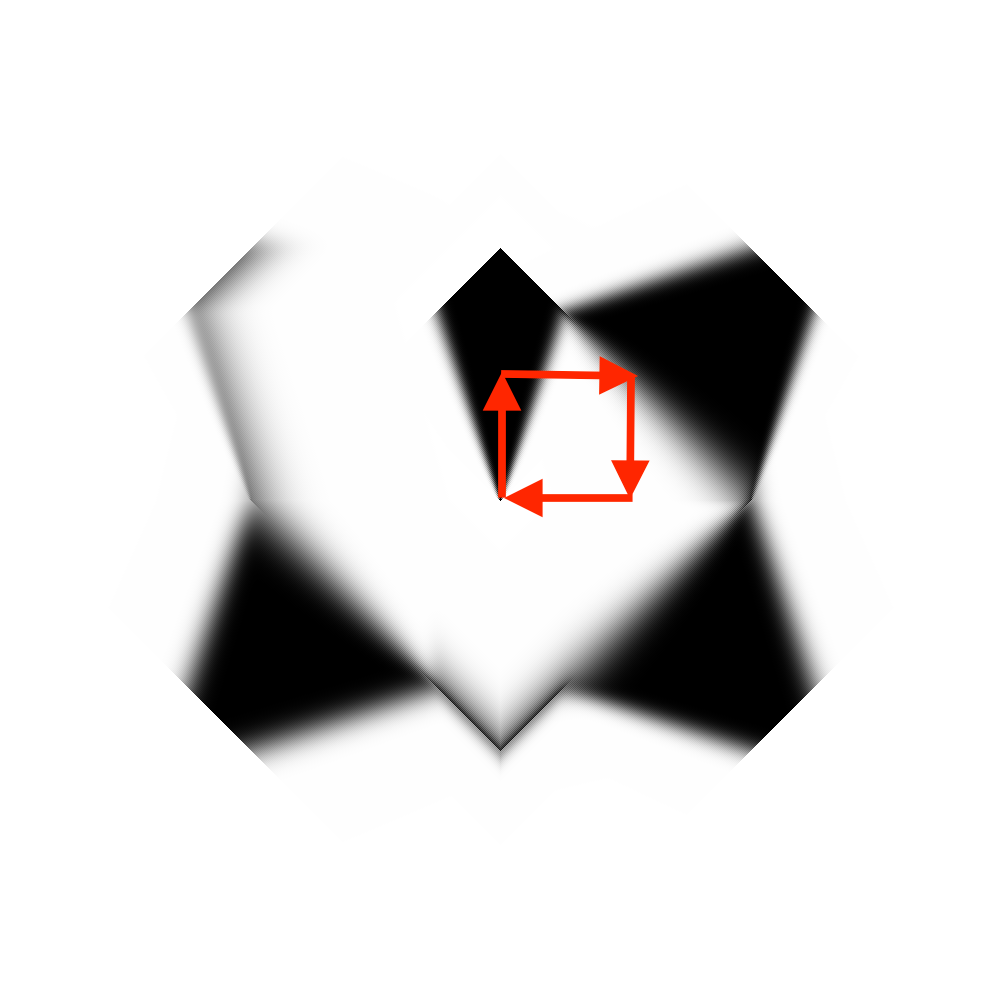} & \includegraphics[width=0.4\textwidth]{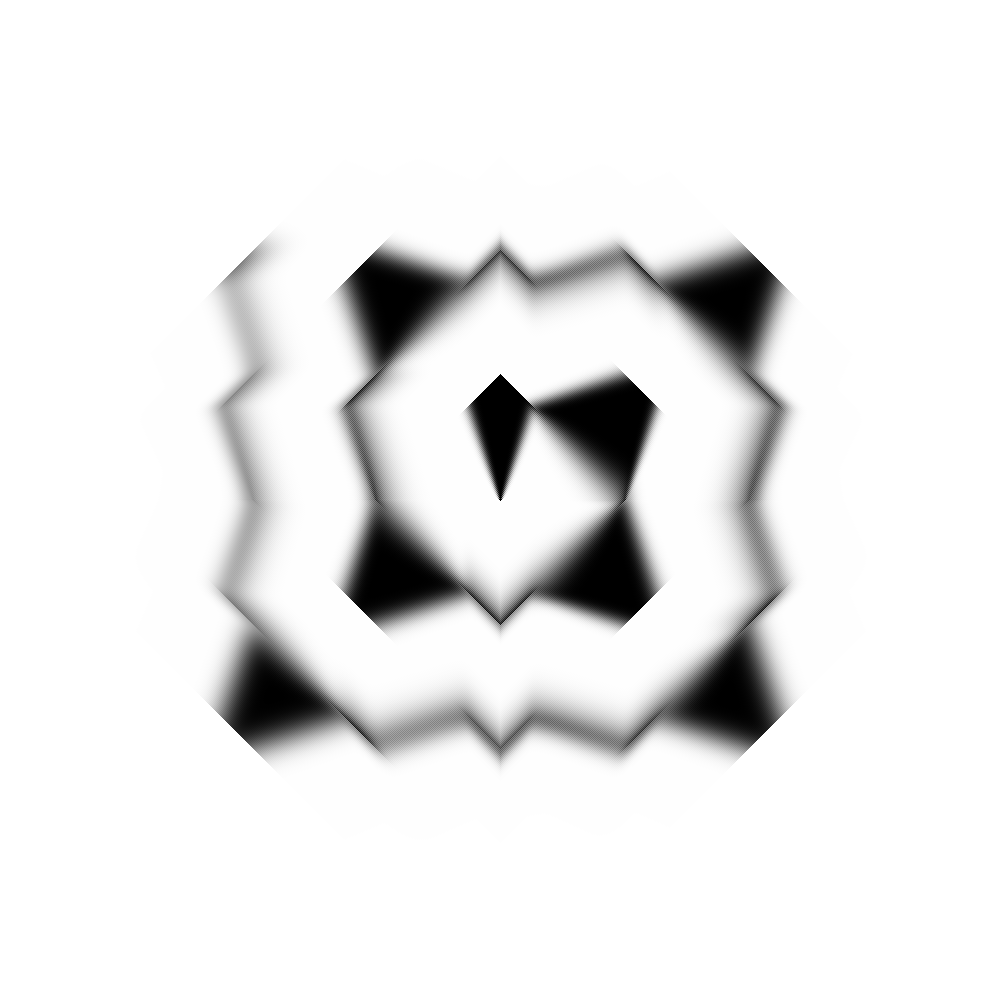} \\ 
(e) one square & (f) two squares \\
\includegraphics[width=0.4\textwidth]{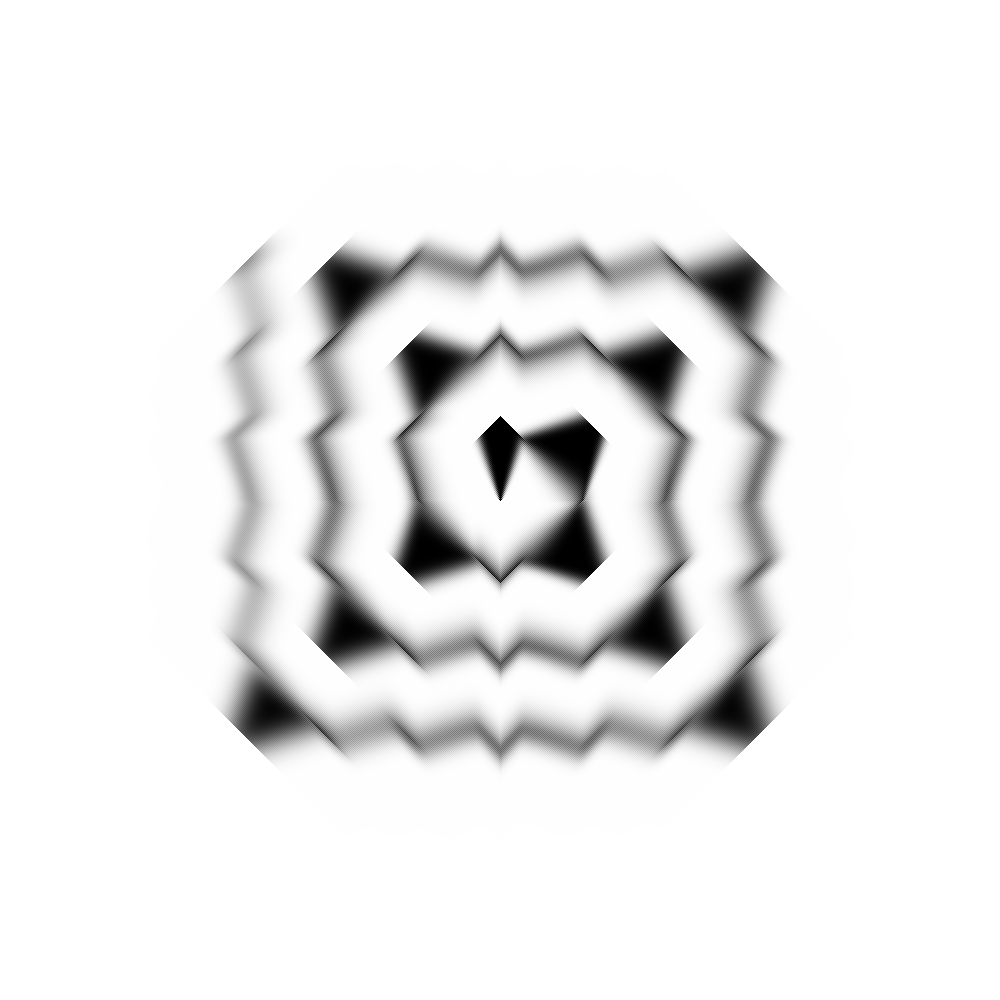} & \includegraphics[width=0.4\textwidth]{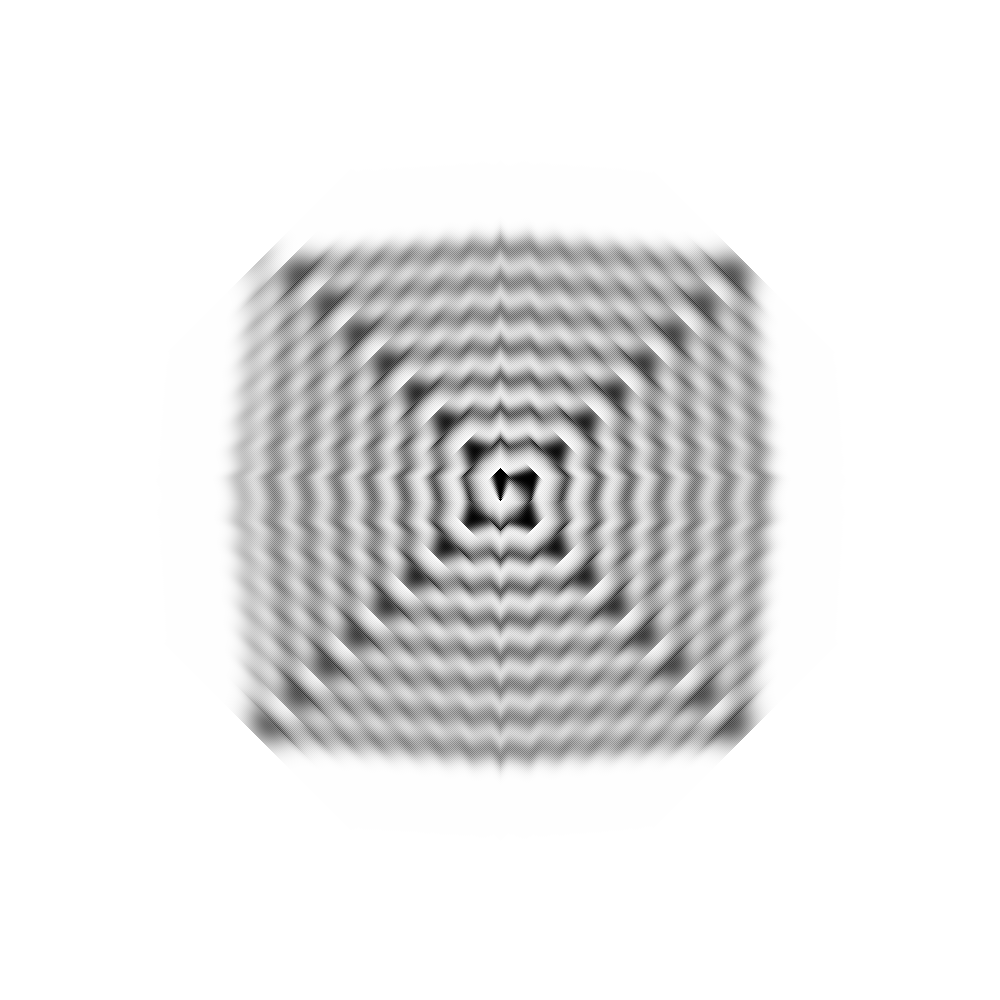} \\ 
(g) three squares & (h) eight squares \\
\end{tabular}
\end{center}
\caption{Examples of 4 more strategies for the first $M$ moves}
\label{fig:4more}
\end{figure}

In this subsection, we present a few natural strategies from family $\mathcal{F}_0$: the survivor $\ldots$
\begin{itemize}
\item [(a)] $\ldots$does not move. 
\item [(b)] $\ldots$performs a random walk.
\item [(c)] $\ldots$goes straight down.
\item [(d)] $\ldots$moves along the diagonal (that is, moves down and then immediately left in each pair of the two consecutive rounds).
\item [(e--h)] $\ldots$goes along edges of a square ($k$ times): $k=1, 2, 3$ or $8$. 
\end{itemize}
For each of them, we performed simulations with various values of $n$ to estimate $t(\mathcal{S})$. As discussed earlier, $p_{\mathcal{S}}(x,y) = 0$ if $|x|+|y| > M= \lfloor n/4 \rfloor$ so it makes sense to use the following scaling: $t(\mathcal{S}) / (n^2/8)$ (note that the number of vertices at distance at most $M$ is asymptotic to $4(M^2/2) \sim n^2 / 8$). In particular, for the worst strategy (from the perspective of the survivor), strategy (a), we get $t(\mathcal{S}) / (n^2/8) \sim 1$. However, for strategies (d) and (h), it seems that different scaling is appropriate, namely, $t(\mathcal{S})/n^{3/2}$. Strategy (d) is analyzed in one of the following subsections and it will become clear why this scaling is more appropriate. We present our results in Table~\ref{tab:simulations}. Finally, functions $p_{\mathcal{S}}(x,y)$ are presented visually on Figures~\ref{fig:4strategies} and~\ref{fig:4more} (dark colours correspond to values of probabilities that are close to 1, light ones to values close to 0).

\begin{table}
\begin{center}
\begin{tabular}{ccccccccc}
$n$ & (a) & (b) & (c) & (d) & (d)' & (e) & (h) & (h)' \\
\hline
1,000 & 1 & 0.699986 & 0.506009 & 0.124296 & 0.491322 & 0.244716 & 0.281038 & 1.110899 \\
2,000 & 1 & 0.694902 & 0.503004 & 0.089770 & 0.501829 & 0.236672 & 0.172459 & 0.964076 \\
4,000 & 1 & 0.692208 & 0.501504 & 0.064498 & 0.509900 & 0.229853 & 0.122445 & 0.968013 \\
8,000 & 1 & 0.690806 & 0.500758 & 0.046374 & 0.518481 & 0.225787 & 0.089200 & 0.997286 \\
\hline
\end{tabular}
\end{center}
\caption{$t(\mathcal{S}) / (n^2/8)$ and $t(\mathcal{S}) / n^{3/2}$ ((d)' and (h)') for various strategies}
\label{tab:simulations}
\end{table}

\subsection{Reduction to one dimensional problem} 

For a while, our conjecture was that $t_n = \Omega(n^2)$ and, as a result, $z(T_n) = O(n)$. In order to simplify the analysis, we may project $p_{\mathcal{S}}(x,y)$ onto $x$; that is, we may concentrate on
$$
q_{\mathcal{S}}^i(x) = \sum_{y = -M}^M p^i_{\mathcal{S}}(x,y).
$$
Clearly, no recursive formula for getting $q_{\mathcal{S}}^{i-1}(x)$ from $q_{\mathcal{S}}^i(x)$ exists but, by setting up a simple coupling, we can get the following lower bound: $q_{\mathcal{S}}^i(x) \ge w^i(x)$, where $w^M(0)=3$, $w^M(-1)=w^M(1)=1$, and for each $2 \le i \le M$ and $1 \le j \le M$
\begin{eqnarray*}
w^{i-1}(x_{i-1}) &=& w^i(x_{i-1}) + 2 \\
w^{i-1}(x_{i-1} \pm j) &=& \frac {w^i(x_{i-1} \pm (j-1) ) + w^i(x_{i-1} \pm j)}{2} + \delta_{j=1},
\end{eqnarray*}
where $\delta_{j=1} = 1$ if $j=1$ and 0 otherwise. (Revisiting Figure~\ref{fig:spreading} might be helpful to see this coupling.) Alternatively, one can apply a slightly weaker coupling to get: $q_{\mathcal{S}}^i(x) \ge z^i(x)$, where $z^M(0)=3$, $z^M(-1)=z^M(1)=1$, and for each $2 \le i \le M$ and $1 \le j \le M$
\begin{eqnarray*}
z^{i-1}(x_{i-1}) &=& z^i(x_{i-1}) + 2 \\
z^{i-1}(x_{i-1} \pm j) &=& \frac {z^i(x_{i-1} \pm (j-1)) + z^i(x_{i-1} \pm j)}{2}.
\end{eqnarray*}
It is easy to show (by induction) that the difference between two consecutive terms, that is $|z^i(x)-z^i(x+1)|$, is always at most 4. Indeed, the property is clearly satisfied for $i=M$; for $2 \le i \le M$ and $1 \le j \le M$ we get
\begin{eqnarray*}
|z^{i-1}(x_{i-1}) - z^{i-1}(x_{i-1} \pm 1) | &=& \left| z^i(x_{i-1}) + 2 - \frac {z^i(x_{i-1}) + z^i(x_{i-1} \pm 1)}{2} \right| \\
&\le& 2 + \frac {| z^i(x_{i-1}) - z^i(x_{i-1} \pm 1) |}{2} \le 4, 
\end{eqnarray*}
and
\begin{align*}
|z^{i-1}(&x_{i-1} \pm j) - z^{i-1}(x_{i-1} \pm (j+1)) | \\
&= \left| \frac {z^i(x_{i-1} \pm (j-1)) + z^i(x_{i-1} \pm j )}{2} - \frac {z^i(x_{i-1} \pm j) + z^i(x_{i-1} \pm (j+1))}{2} \right| \\
&= \left| \frac {z^i(x_{i-1} \pm (j-1)) - z^i(x_{i-1} \pm j )}{2} + \frac {z^i(x_{i-1} \pm j) - z^i(x_{i-1} \pm (j+1))}{2} \right| \\
&\le  \frac {\left| z^i(x_{i-1} \pm (j-1)) - z^i(x_{i-1} \pm j ) \right|}{2}  + \frac {\left| z^i(x_{i-1} \pm j) - z^i(x_{i-1} \pm (j+1))\right|}{2} \le 4.
\end{align*}
In particular, it follows that the sequence $(z^i(x_{i}))_{i=M}^1$ goes up by 2 if the survivor ``pauses'' (with respect to this projection; that is, when $x_i = x_{i+1}$) and goes down by at most 2 if he ``moves'' (that is, when $x_i \neq x_{i+1}$). Hence, if a strategy $\mathcal{S}$ has the property that at some point $t \le M/2$ there were $s$ more rounds when the survivor pauses than moves, then $t(\mathcal{S}) = \Omega(s^2)$. Since, without loss of generality, we may assume that the survivor pauses at for at least $M/2$ rounds (by projecting onto $y$ instead of $x$ if needed), this ``almost'' imply some nontrivial bound for $t_n$. For a while, we were hoping to be able to show that there exists an $\eps > 0$ such that $z^1(x_1) \ge \eps n$ which would imply $t_n = \Theta(n^2)$ and this in turn would imply $z(T_n) = O(n)$. Unfortunately, it is not true. Based on simulations, we were able to identify that strategy (d) might create a problem and serve as a counterexample which turned out to be the case. We will show in the next section that the conjecture was too optimistic and, in fact, $t_n = O(n^{3/2})$. But there is still a possibility that $t_n = \Theta(n^{3/2})$ and so perhaps $z(T_n) = O(n^{3/2})$.

\subsection{$t(\mathcal{S}) = O(n^{3/2})$ for strategy (d).} 

Consider strategy (d) discussed above; that is, suppose that the survivor starts at vertex $(\lfloor M/2 \rfloor, \lceil M/2 \rceil)$, where $M = \lfloor n/4 \rfloor$, and goes to $(0,0)$ along the diagonal (that is, moves ``South'' and then immediately ``West'' in each pair of the two consecutive rounds). We will show the following result:

\begin{theorem}\label{thm:strategy_d}
$t(\mathcal{S}) = O(n^{3/2})$ for strategy (d).
\end{theorem}
\begin{proof}
In order to estimate $t(\mathcal{S})$, we need to estimate $p_{\mathcal{S}}(x,y)$ for each $x,y\in \Z$ such that $|x|+|y| \le M$. Due to the symmetry, we may assume that $x \ge y$. Suppose that a given zombie starts the game at vertex $(x,y)$. We will partition the part of the grid we investigate into 4 regions and deal with each of them separately. See Figure~\ref{fig:strategy_d}. Let $c \in \R_+$ be some fixed, large enough, constant.

\begin{figure}
\begin{center}
\includegraphics[width=0.4\textwidth]{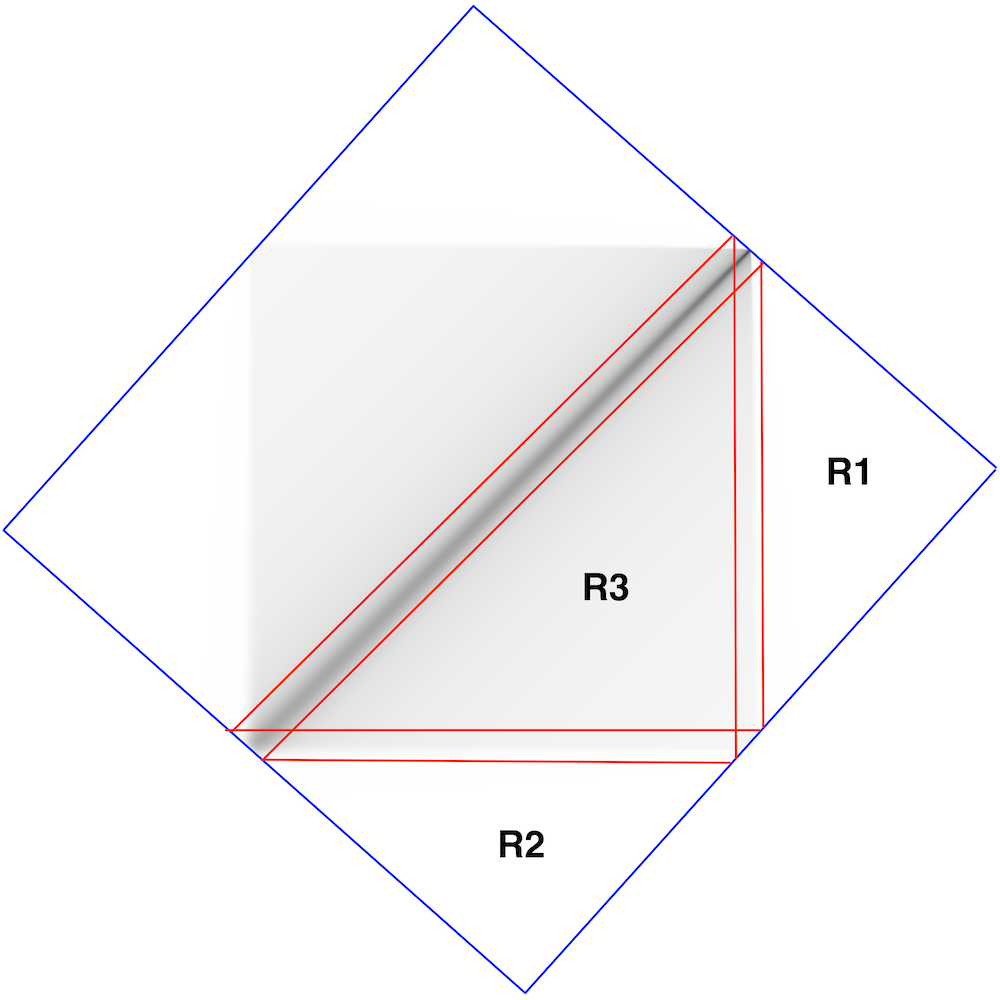}
\end{center}
\caption{Strategy (d) under microscope}
\label{fig:strategy_d}
\end{figure}

\medskip Region \textbf{R1}: Suppose that $x \ge M/2 + c \sqrt{n}$. The zombie moves randomly ``North'' or ``West'' until her position and a position of the survivor match horizontally or vertically. If they match horizontally, she will never be able to catch the survivor---the distance will be preserved till the end of the game. On the other hand, if coordinates are matched vertically, then there is a chance. The probability of this event can be estimated (see discussion for Region~\textbf{R3} below) but we do not need it here (we may use a trivial upper bound of 1 for this conditional probability). It is enough to notice that if the vertical match occurs, then at some point of the game a zombie moved $x-\lfloor M/2 \rfloor$ more often ``West'' than she moved ``North''. Hence, 
$$
p_{\mathcal{S}}(x,y) \le \Prob \Big( S_t \ge t/2 + (x-\lfloor M/2 \rfloor) \text{ for some } t \le M \Big), 
$$
where $X_1, X_2, \ldots$ is a sequence of independent random variables, each of them being the Bernoulli random variable with parameter $p=1/2$, and $S_t = \sum_{i=1}^t X_t$. It follows that
$$
p_{\mathcal{S}}(x,y) \le \exp \left( - \Omega \left( \frac { (x-\lfloor M/2 \rfloor)^2 }{M} \right) \right). 
$$
(For example by converting $S_t$ into a martingale and using Hoeffding-Azuma inequality.) The contribution to $t(\mathcal{S})$ from Region~\textbf{R1} is then at most 
$$
\sum_{x \ge M/2 + c \sqrt{n}} O(n) \exp \left( - \Omega \left( \frac { (x-\lfloor M/2 \rfloor)^2 }{M} \right) \right) =
O(n) \sum_{x \ge c \sqrt{n}}  \exp \left( - \Omega \left( \frac { x^2 }{n} \right) \right) = O(n^{3/2}),
$$
provided that $c$ is large enough. 

\medskip Region \textbf{R2}: Suppose now that $y \le -M/2 - c \sqrt{n}$. The zombie moves ``North'' or ``East'', and it is expected that players match horizontally after $r$ rounds at which point the zombie occupies vertex $(\hat{x}, \hat{y})$, where
\begin{eqnarray*}
\hat{x} &=& \frac {M/2+x}{2} + O(1) = \frac {M}{4} + \frac {x}{2} + O(1)\\
r &=& 2 ( \hat{x} - x)  + O(1) = 2 \left( \frac {M}{4} - \frac {x}{2} \right)  + O(1)\\
\hat{y} &=& y+ \frac {r}{2} = y + \left( \frac {M}{4} - \frac {x}{2} \right)  + O(1) = \frac {M}{4} - \frac {x}{2} + y + O(1).
\end{eqnarray*}
The distance from $(\hat{x}, \hat{y})$ to $(0,0)$ is $d = \hat{x} - \hat{y} = x - y + O(1)$ and so $r+d = M/2 - y + O(1) \ge M + c \sqrt{n} + O(1)$. It follows that in order for the zombie to have a chance to win, horizontal match has to occur much later than expected. Arguing as before, we can estimate the probability of this event and show that the contribution to $t(\mathcal{S})$ from Region~\textbf{R2} is $O(n^{3/2})$, provided that $c$ is large enough. 

\medskip Region \textbf{R3}: Suppose now that $x \ge y + c \sqrt{n}$, $y \ge -M/2 + c \sqrt{n}$, and $x \le M/2 - c \sqrt{n}$. This case seems to be the most interesting. As for the previous region, it is expected that players match horizontally when the zombie occupies vertex $(\hat{x}, \hat{y})$ and the distance between players is
$$
k = \left( \frac M2 - y \right) - \left( \frac M2 - x \right) = x-y \ge c \sqrt{n}. 
$$
Arguing as before, we can show that with probability $1-\exp(-\Omega(-k^2/M))$ not only this happens but at that time the distance between players, $Y$, is at least $k/2$. 

Conditioning on $Y = y \ge k/2$, we aim now to estimate the probability that the survivor is eaten. Assume that $y$ is even; the odd case can be dealt similarly. Consider a sequence of two consecutive rounds. At the beginning of each pair of rounds, before the survivor goes ``South'', we measure the absolute difference between the corresponding $x$-coordinates of the players, to get a sequence $Z_0, Z_1, \ldots, Z_{y/2}$ of random variables. Clearly, $Z_0=0$ and the survivor is eaten if and only if $Z_{y/2}=0$. Indeed, if $Z_{y/2} > 0$, then the zombie ends up lined up horizontally before getting close to the survivor and from that point on she will continue keeping the distance. If $Z_t > 0$, then, 
$$
Z_{t+1} =
\begin{cases}
Z_t + 1 & \text{ with probability $1/4$ (zombie goes ``North'' twice)} \\
Z_t - 1 & \text{ with probability $1/4$ (zombie goes ``West'' twice)} \\
Z_t & \text{ with probability $1/2$ (zombie goes once ``West'' and once ``North'')}.
\end{cases}
$$
On the other hand, if $Z_t = 0$, then the first move of the zombie is forced (she goes ``North'') and so 
$$
Z_{t+1} =
\begin{cases}
1 & \text{ with probability $1/2$ (the second move is ``North'')} \\
0 & \text{ with probability $1/2$ (the second move is ``West'')}.
\end{cases}
$$
We can couple this process with a lazy random walk and one can show that $\Prob(Z_{y/2} = 0) = \Theta(1/\sqrt{y}) = \Theta(1/\sqrt{k})$. The contribution to $t(\mathcal{S})$ from Region~\textbf{R3} is 
$$
O(n) \sum_{k = c \sqrt{n}}^{O(n)} \left( \exp \left( - \Omega \left( \frac { k^2 }{n} \right) \right) + \frac {1}{\sqrt{k}} \right) = O(n^{3/2}) + O(n) \int_{\sqrt{n}}^{O(n)} \frac {dx}{x} = O(n^{3/2}),
$$
provided that $c$ is large enough. 

\medskip The number of vertices not included in the three Regions we considered is $O(n^{3/2})$ and so the total contribution is $O(n^{3/2})$ and the proof is finished.
\end{proof}

\subsection{Potential improvement on the lower bound: $z(T_n) = \sqrt{n}/\omega$, where $\omega = \omega(n)$ is any function going to infinity as $n \to \infty$.}

Let $\omega = \omega(n)$ be any function going to infinity as $n \to \infty$. Suppose that the survivor uses strategy (d) (regardless what zombies are doing). This time, he continues moving ``diagonally'' forever (of course, unless he is eaten earlier). In the previous section, in order to avoid problems with a boundary effect, we restricted ourselves to sub-graph around $(0,0)$. However, it is straightforward to extend the argument and show that if the survivor is not eaten for long enough, all the zombies will stay behind him keeping their distances forever. We do not do it here as the improvement would be minor anyway. One can show that, a.a.s.\ $k = \sqrt{n} / \omega$ zombies cannot catch the survivor on $T_n$ and so $z(T_n) > \sqrt{n} / \omega$. Indeed, after extending the argument, the probability that no zombie catches the survivor would be
$$
\left( 1 - \frac {t(\mathcal{S})}{n^2} \right)^k = \exp \left( - O \left( \frac {k}{\sqrt{n}} \right) \right) = \exp ( -o(1)) \sim 1.
$$ 
This would only be a small improvement comparing to the lower bound in~\cite{zombies} and so we are not formalize the argument here. Is there a lower bound of $n^{1/2 + \eps}$ for some $\eps>0$? This remains an open question.


\begin{thebibliography}{99}

\bibitem{zombies} A.\ Bonato, D.\ Mitsche, X. Per\'ez-Gimen\'ez, and P.\ Pra\l{}at, A probabilistic version of the game of Zombies and Survivors on graphs, \emph{Theoretical Computer Science} \textbf{655} (2016), 2--14.

\bibitem{hm} S.L.\ Fitzpatrick, J.\ Howell, M.E.\ Messinger, D.A.\ Pike, A deterministic version of the game of zombies and survivors on graphs, Preprint 2015.


\end{thebibliography}
\end{document}